\documentclass[twosided]{amsart}
\usepackage{amsmath}
\usepackage{amsfonts}
\usepackage{amssymb,enumerate}
\usepackage{amsthm}
\usepackage{hyperref}
\usepackage{centernot}
\usepackage{stmaryrd}
\usepackage{relsize}
\usepackage{setspace}

\DeclareMathOperator\Spec{\displaystyle{Spec_{r}}}
\DeclareMathOperator\supp{supp}
\theoremstyle{plain}
\newtheorem{theorem}{Theorem}[section]
\newtheorem{lemma}[theorem]{Lemma}
\newtheorem{definition}[theorem]{Definition}
\newtheorem{corollary}[theorem]{Corollary}
\newtheorem{proposition}[theorem]{Proposition}
\newtheorem{remark}[theorem]{Remark}

\begin{document}

\title[Discrete orderings in the real spectrum]{Discrete orderings in the real spectrum}
\author{Shahram Mohsenipour}

\address{School of Mathematics, Institute for Research in Fundamental Sciences (IPM)
        P. O. Box 19395-5746, Tehran, Iran}\email{sh.mohsenipour@gmail.com}

\thanks{The research of the first author was in part supported by a grant from IPM (No. 97030403)}

\begin{abstract} We study discrete orderings in the real spectrum of a commutative ring by defining discrete prime cones and give an algebro-geometric meaning to some kind of diophantine problems over discretely ordered rings. Also for a discretely ordered ring $M$ and a real closed field $R$ containing $M$ we prove a theorem on the distribution of the discrete orderings of $M[X_1,\dots,X_n]$ in $\Spec(R[X_1,\dots,X_n])$ in geometric terms. To be more precise, we prove that any ball $\mathbb{B}(\alpha,r)$ in $\Spec(R[X_1,\dots,X_n]$) with center $\alpha$ and radius $r$ (defined via Robson's metric) contains a discrete ordering of $M[X_1,\dots,X_n]$ whenever $r$ is non-infinitesimal and $\alpha$ is away from all hyperplanes over $M$ passing through the origin.
\end{abstract}
\maketitle
\bibliographystyle{amsplain}
\section{introduction}
Let $(M,<)$ be an ordered commutative ring with identity. We say that $(M,<)$ is a {\em discretely ordered ring}, or equivalently, the order $<$ is {\em discrete}, if there is no $a\in M$ with $0<a<1$. Two basic examples are the ring of integers $\mathbb{Z}$ with its usual ordering and the polynomial ring $\mathbb{Z}[X_1,\dots,X_n]$ ordered in such a way that $X_1$ is infinitely large and $X_{i+1}$ is greater than every element of $\mathbb{Z}[X_1,\dots,X_i]$, $i=1,\dots,n-1$. It is easily seen that the ring of integers is a convex subring of any discretely ordered ring. We are interested in some kind of diophantine problems which can be formulated in general terms as follows. Let $M$ be a discretely ordered ring and $f(X_1,\dots,X_n)\in M[X_1,\dots,X_n]$, we want to know whether the equation $f(X_1,\dots,X_n)=0$ has a solution in a discretely ordered ring $N$ extending $M$ which means that $M$ is a subring of $N$ and the restriction of the order of $N$ to $M$ is the same as the order of $M$. We also say that $N$ is a discrete extension of $M$.

This type of problems have already been considered in mathematical logic in the model theoretic study of weak systems of arithmetic. For instance it is proved in \cite{otero} that if $M$ is a discretely ordered ring and $a\in M$ such that $a$ is infinite, i.e., $a>n$ for all $n\in\mathbb{Z}$, then the equation $x^2+y^2=a$ has a solution in a discrete extension of $M$. Also in an important result and by using the arithmetic theory of curves, van den Dries \cite{vandendries} has shown that there is a decision procedure which for a given $f(X,Y)\in\mathbb{Z}[X,Y]$ decides whether $f(X,Y)=0$ has a solution in a discretely ordered ring. The existence of a general decision procedure for $f(X_1,\dots,X_n)=0$ with $f(X_1,\dots,X_n)\in \mathbb{Z}[X_1,\dots,X_n]$ is major open problem in this field posed by Wilkie \cite{wilkie}.

Having regarded discrete orderings as arithmetical objects and also motivated by Arithmetic Geometry, we give a geometric meaning to the above diophantine problems in the first part of the paper. Our main tool for doing this is the real spectrum, coming from the real algebraic geometry. Let $A$ be a commutative ring with the identity, the real spectrum of $A$, denoted by $\Spec(A)$, is a topological space consisting of all pairs $(\mathfrak{p},<)$ such that $\mathfrak{p}$ is a real prime ideal of $A$ and $<$ is an ordering of the quotient field of $A/\mathfrak{p}$. Any such pair $(\mathfrak{p},<)$ corresponds to a prime cone $\alpha$ of $A$ with $\supp(\alpha)=\mathfrak{p}$. We shall define a {\em discrete prime cone} of $A$ and show that the discrete prime cones of $A$ are in bijective correspondence with those pairs $(\mathfrak{p},<)$ such that $<$ is a discrete ordering of the ring $A/\mathfrak{p}$. Now let $M$ be discretely ordered subring of $A$, we define an {\em M-discrete prime cone} $\alpha$ of $A$ as a discrete prime cone of $A$ such that $\alpha\cap M$ becomes the positive cone of $M$. We also show that $M$-discrete prime cones of $A$ are in bijective correspondence with those elements $\alpha$ of $\Spec(A)$ such that there is an injective order preserving map from $M$ into $k(\supp(\alpha))$, where $k(\supp(\alpha))$ denotes the quotient field of $A/\supp(\alpha)$. Coming back to our diophantine problems, let $R$ be a real closed field and $M$ be discrete subring of $R$. We say that $\alpha\in\Spec(R[X_1,\dots,X_n])$ is an $M$-{\em arithmetical point} of $\Spec(R[X_1,\dots,X_n])$ if $\alpha\cap M[X_1,\dots,X_n]$ is an $M$-discrete prime cone of $M[X_1,\dots,X_n]$. Now we can present our geometric counterpart of the diophantine problem stated above as follows. The equation $f(X_1,\dots,X_n)=0$ has a solution in a discretely ordered ring $N$ extending $M$ iff there is an $M$-arithmetical point of $\widetilde{V}\subseteq\Spec(R[X_1,\dots,X_n])$, where $V$ is the algebraic subset of $R^n$ defined by $f(X_1,\dots,X_n)=0$ and $\widetilde{V}$ is obtained by Coste-Roy tilde operation \cite{BCR}. This also will enable us to discuss the discrete orderings of the polynomial ring $M[X_1,\dots,X_n]$ in geometric terms. In fact the discrete orderings of $M[X_1,\dots,X_n]$ are in bijective correspondence with those  $M$-arithmetical points of $\Spec(R[X_1,\dots,X_n])$  which are transcendental, namely, they do not lie on any $\widetilde{V}\subseteq\Spec(R[X_1,\dots,X_n])$ where $V$ is a proper real algebraic subset of $R^{n}$ defined over $M$.

In the second part of the paper, which is the main part, we prove a theorem on the distribution of transcendental $M$-arithmetical points in $\Spec(R[X_1,\dots,X_n])$ in geometric terms. In the sequel, we assume familiarity with the basic facts about the real spectrum (See Chapter 7 of \cite{BCR}). Let $\bar{a}=(a_1,\dots,a_n)\in M^{n}$ and $r\in R^{>0}$. Let $H_{\bar{a}}$ be the hyperplane in $R^n$ defined by
\[
H_{\bar{a}}=:\{x\in R^n\,|\,\sum_{i=1}^{n} a_i x_i=0\}.
\]
Let $\alpha$ be a transcendental $M$-arithmetical point of $\Spec(R[\bar{X}])$. This implies that $M[X_1(\alpha),\dots,$ $X_n(\alpha)]$ is a discretely ordered subring of $k(\alpha)$ (the real closure of the residue field $k(\supp\alpha)$) and for $a_1,\dots,a_n$ in $M$ we have that $a_1 X_1(\alpha)+\dots+a_n X_n(\alpha)\neq0$ holds true in $k(\alpha)$. Working in $ \Spec(R[\bar{X}])$, it is equivalent to saying that $\alpha$ doesn't lie on the hyperplane $\widetilde{H}_a$. Let $\mathbb{B}(\alpha,r)$ be the ball with center $\alpha$ and radius $t$ defined by the metric $\mu$ introduced by Robson \cite{robson}. In the main theorem of this paper, Theorem \ref{main}, we prove that {\em if $\mathbb{B}(\alpha,t)\cap\widetilde{H}_{\bar{a}}=\emptyset$ for every $\bar{a}\in M^{n}$ and every finite $t\in R$, then every ball $\mathbb{B}(\alpha,r)$ with a non-infinitesimal radius $r$ contains a transcendental $M$-arithmetical point.}

We now give a brief sketch of how the proof proceeds. We work in $k(\alpha)^n$. Let $S_{\alpha,r}\subseteq k(\alpha)^n$ be the Euclidean ball with center $(X_1(\alpha),\dots,X_n(\alpha))$ and radius $r$. We first show that for any finite set of non-constant polynomials $F_1(\bar{X}),\dots, F_l(\bar{X})$ in $M[\bar{X}]$, there is a positive integer $N$ and a suitable vector $\overrightarrow{q}\in\mathbb{Q}^n$ such that for any $N$ successive points $P_1,\dots,P_N$ in $S_{\alpha,r}$, lying on the line $L$ passing through $P_1$ and parallel to $\overrightarrow{q}$ with the following properties: (1) $P_1=(X_1(\alpha),\dots,X_n(\alpha))$, (2) each $||\overline{P_iP_{i+1}}||$ is finite and non-infinitesimal, we have that there is $P\in\{P_1,\dots,P_N\}$ such that all $F_1(P),\dots,F_n(P)$ are infinite. This will be done through a reduction process  by successive using of the higher-dimensional mean value theorem for real closed fields and also choosing the $\mathbb{Q}$-rational point $\bar{q}$ from a suitable Zariski open subset of the affine space $\mathbb{A}^n_{k(\alpha)}$. Having shown $\pi(\widetilde{S_{\alpha,r}})\subseteq\mathbb{B}(\alpha,r)$, we can come back into $\mathbb{B}(\alpha,r)\subseteq\Spec(R[\bar{X}])$, where $\pi$ is the canonical projection
\[
\pi\colon\Spec(k(\alpha)[X_1,\dots,X_n])\longrightarrow\Spec(R[X_1,\dots,X_n])
\]
and then by using a compactness argument, we can find $\gamma$ in $\mathbb{B}(\alpha,r)$ such that for every non-constant polynomial $F(\bar{X})$ in $M[\bar{X}]$, $F(X_1(\gamma),\dots,X_n(\gamma))$ is infinite in $k(\alpha)$. It will follow that $M[X_1(\gamma),\dots,X_n(\gamma)]$ is a discretely ordered ring, so $\gamma$ is a transcendental $M$-arithmetical point.

It is worth mentioning that the starting point of this work was Boughattas' ``real" proof \cite{boughattas} of Wilkie's extension theorem \cite{wilkie} in model theory of arithmetic which for some time we have been thinking it was potentially capable of saying something about the distribution of discrete orderings in some suitable space of orderings. Wilkie's theorem asserts that if $K$ is an $|R|^{+}$-saturated extension of a real closed field $R$ and $\alpha\in K$ has infinite distance from every $\frac{a}{n}$ where $a\in M$, $n\in\mathbb{Z}$, then there is $\beta\in K$ such that $\beta$ has finite distance from $\alpha$ and $M[\beta]$ is a discretely ordered ring. Wilkie's original proof uses algebraically closed fields. In fact Theorem \ref{main} and its proof can be seen as a higher dimensional generalization of Wilkie's theorem and Boughattas' proof in terms of the real algebraic geometry.

\section{Preliminaries}\label{prelim}
In this paper we assume familiarity with basic notions of the real spectrum (Chapter 7, \cite{BCR}) and also basic model theory of real closed fields at the level of Tarski's transfer theorem. However we review some basic definitions and make some conventions. By ring we always mean a commutative ring with identity. Let $A$ be a domain, $F(A)$ will denote the fraction field of $A$. For a ring $A$ we denote the real spectrum of $A$ by $\Spec(A)$ and for $\alpha\in\Spec(A)$ we denote the residue field of $\alpha$ by $F(A/\supp(\alpha))=k(\supp(\alpha))$. The prime cone $\alpha$ induces an ordering $<_{\alpha}$ on $k(\supp(\alpha))$. We denote its real closure by $k(\alpha)$. We sometimes denote $k(\alpha)$ by $k_R(\alpha)$ when we want to emphasis on the base field $R$. Let $f\in A$ and $\alpha\in\Spec(A)$, then $f(\alpha)$ denotes the image of $f$ under the canonical map $A\longrightarrow A/\supp(\alpha)$. Let $R$ be a real closed field and $\mathcal{S}\subseteq R^n$ be a semialgebraic set. We denote the formula defining $\mathcal{S}$ again by $\mathcal{S}$ and define
\[
\widetilde{S}:=\big{\{}\alpha\in\Spec(R[X_1,\dots,X_n])|\,k(\supp(\alpha))\models\mathcal{S}(X_1(\alpha),\dots,X_n(\alpha)) \big{\}}.
\]
For a real closed field $K$ containing $R$, we denote the $K$-rational points of $\mathcal{S}$ by $\mathcal{S}(K)$, which is
\[
\mathcal{S}(K):=\{\bar{x}\in K^n|\,K\models\mathcal{S}(x_1,\dots,x_n)\}.
\]

We say that an element $a$ of a (non-archimedean) real closed field $R$ is {\em finite} if there is a positive integer $n$ such that $|a|<n$, otherwise it is called {\em infinite}. A finite element $a\in R$ is called {\em infinitesimal}, if for every positive integer $n$, we have $|a|<\frac{1}{n}$. We denote the subset of all finite elements of $R$ by $R_{fin}$.

We denote the set of non-negative integers by $\omega$, also $A^n$ denotes the set of n-tuples with entries from $A$ and $\omega^{n}_{*}=\omega^{n}-\{\bar{0}\}$. Let $\bar{a},\bar{b}\in\omega^{n}$, then $\bar{a}+\bar{b}\in\omega^{n}$ is the pointwise sum of $\bar{a}$ and $\bar{b}$.

For any two points $P$ and $Q$ in a Euclidean space, we denote the line passing through $P$ and $Q$ by $PQ$. Also $\overline{PQ}$, $||\overline{PQ}||$ and $\overrightarrow{PQ}$ will denote, respectively, the segment with the endpoints $P$ and $Q$, its length and the vector with the tail $P$ and the head $Q$.

We now review what we need from Robson's metric \cite{robson}. Let $R$ be a real closed field and $\alpha,\beta\in\Spec(R[X_1,\dots,X_n])$. Let $A,B$ be two semialgebraic subsets of $R^n$ with $\alpha\in\widetilde{A}$, $\beta\in\widetilde{B}$ we define
\[
d(A,B):= \inf_{\substack{\bar{a}\in A\\ \bar{b}\in B}} ||\bar{a}-\bar{b}||,
\]
in which the infimum is to be taken in $\Spec(R[X])$ as explained in \cite{robson}. Now we define
\[
\mu(\alpha,\beta):= \sup_{\substack{\alpha\in\widetilde{A}\\ \beta\in\widetilde{B}}} d(A,B),
\]
again the supremum is taken in $\Spec(R[X])$ in a similar fashion. So we have the following function
\[
\mu\colon\Spec(R[\bar{X}])\times\Spec(R[\bar{X}])\rightarrow\Spec(R[X]).
\]
\begin{theorem}[Robson \cite{robson}, Theorem II]
The function $\mu$ is a positive definite symmetric function satisfying the triangle inequality. Moreover, if $\bar{a},\bar{b}\in R^n\subseteq\Spec(R[\bar{X}])$, then $\mu(\bar{a},\bar{b})=||\bar{a}-\bar{b}||$.
\end{theorem}
Now we define the Robson ball $\mathbb{B}(\alpha,r)$ with center $\alpha\in\Spec(R[X_1,\dots,X_n])$ and radius $r\in R$ as
\[
\mathbb{B}(\alpha,r):=\{\beta\in\Spec(R[X_1,\dots,X_n])|\,\mu(\alpha,\beta)\leq r\}.
\]

\section{discrete prime cones}
For rings $A,B$, by $A\subseteq B$, we mean $A$ is a subring of $B$. For an ordered ring $(A,<)$, we denote its positive cone by $A^{\geq0}$. We also sometimes talk about an ordered ring, say $A$, without mentioning its order.

\begin{definition}\label{mo}
Let $A$ be a ring. We say that $\alpha\in \Spec(A)$ is a \textsl{discrete prime cone} of $A$ if for every $f\in A$, if $f\in\alpha$, $f\notin \supp(\alpha)$, then $f-1 \in\alpha$.
\end{definition}
\begin{lemma}\label{1}
Let $A$ be a ring and $\alpha\in \Spec(A)$, then $\alpha$ is a discrete prime cone of $A$ iff $A/\supp(\alpha)$ is a discretely ordered ring.
\end{lemma}
\begin{proof}
Let $\alpha$ be a discrete prime cone of $A$ and $f\in A$. Let $\bar{f}$ be the image of $f$ under the canonical homomorphism $A\rightarrow A/\supp(\alpha)$. Suppose that $\bar{f}>0$ in $A/\supp(\alpha)$. Then $f\in\alpha$ and $f\notin \supp(\alpha)$, so $f-1 \in\alpha$, so $\bar{f}-1\geq 0$ in $A/\supp(\alpha)$. Therefore $A/\supp(\alpha)$ is a discretely ordered ring.

Conversely let $A/\supp(\alpha)$ be a discretely ordered ring. Suppose that $f\in A$, $f\in\alpha$, $f\notin\supp(\alpha)$. It follows that $\bar{f}>0$ in $A/\supp(\alpha)$. Also discreteness of $A/\supp(\alpha)$ implies that $\bar{f}-1\geq0$. So $f-1 \in\alpha$. Thus $\alpha$ is a discrete prime cone of $A$.
\end{proof}
\begin{lemma} \label{2}
Let $\mathfrak{p}$ be a real prime ideal of $A$ such that $A/\mathfrak{p}$ has a discrete ordering. Then there is $\alpha\in \Spec(A)$ such that $\alpha$ is a discrete prime cone of $A$.
\end{lemma}
\begin{proof}
Let $(A/\mathfrak{p},<)$ be a discretely ordered ring. Consider the canonical homomorphism $\varphi\colon A\rightarrow A/\mathfrak{p}$. Set $\alpha=\varphi^{-1}((A/\mathfrak{p})^{\geq 0})$. It is easy to see that $\alpha$ is a prime cone of $A$. For discreteness, we first note that $\supp(\alpha)=\mathfrak{p}$. Now suppose $a\in\alpha$, $a\notin\supp(\alpha)$, so $a\notin\mathfrak{p}$, then $\varphi(a)\in(A/\mathfrak{p})^{> 0}$. Since $(A/\mathfrak{p},<)$ is a discretely ordered ring, we have $\varphi(a-1)=\varphi(a)-1\in(A/\mathfrak{p})^{\geq 0}$ which implies $a-1\in\alpha$.
\end{proof}

Now from Lemmas \ref{1}, \ref{2} we conclude

\begin{corollary} There is a bijective correspondence between the set of all discrete prime cones of $A$ and the set of all pairs $(\mathfrak{p},<)$ where $\mathfrak{p}$ is a real prime ideal of $A$ and $<$ is a discrete ordering of $A/\mathfrak{p}$. Also as a special case, there is a bijective correspondence between the set of all discrete orderings of $A$ and the set of all discrete prime cones $\alpha$ of $A$ such that $\textrm{supp}(\alpha)=\{0\}$.
\end{corollary}
\begin{definition} Let $(A,<)$, $(B,<)$ be ordered rings, $(B,<)$ is called an \textsl{ordered extension} of $(A,<)$ if $A$ is a subring of $B$ and the restriction of the ordering of $B$ to $A$ is the same as the ordering of $A$. In this case we also say that $A$ is an ordered subring of $B$.
\end{definition}
Now we move towards the $M$-relativized version of the above theorems where $M$ is a discretely ordered ring. Note that in the definition below, $M$ is just an ordered ring.
\begin{definition}\label{mo}
Let $M$ be an ordered ring and let $A$ be a ring with $M\subseteq A$. We say that $\alpha\in \Spec(A)$ is an $M$-\textsl{prime cone} of $A$ if
\begin{itemize}
\item[(i)] $M^{\geq0}\subseteq\alpha$,
\item[(ii)] $M^{\geq0}\cap \textrm{supp}\,(\alpha)=\{0\}$.
\end{itemize}
\end{definition}
It is clear from the definition above that if $\alpha\in \Spec(A)$ is an $M$-prime cone of $A$, then there is an order-preserving injective homomorphism $M\rightarrow A/\supp(\alpha)$, so we can regard $A/\supp(\alpha)$ as an ordered extension of (an isomorphic copy) of $M$. For simplicity we will identify $M$ and its image in $A/\textrm{supp}(\alpha)$.
\begin{definition}\label{dpc}
Let $(M,<)$ be a discretely ordered ring and let $A$ be a ring with $M\subseteq A$. We say that $\alpha\in \Spec(A)$ is an $M$-\textsl{discrete prime cone} of $A$ if
\begin{itemize}
\item[(i)] $M^{\geq0}\subseteq\alpha$,
\item[(ii)] $M^{\geq0}\cap \supp(\alpha)=\{0\}$,
\item[(iii)] for every $f\in A$, if $f\in\alpha$, $f\notin \supp(\alpha)$, then $f-1 \in\alpha$.
\end{itemize}
\end{definition}

\begin{definition} For a discretely ordered ring $(M,<)$ and an ordered ring $(N,<)$ we say that $N$ is a \textsl{discrete extension} of $M$, if $N$ is an ordered extension of $M$ and $N$ is a discretely ordered ring .
\end{definition}
\begin{lemma}
Let $(M,<)$ be a discretely ordered ring and let $A$ be a ring with $M\subseteq A$. Suppose that $\alpha\in \Spec(A)$ is an $M$-discrete prime cone of $A$, then  $A/\supp(\alpha)$ is a discrete extension of $M$. Conversely, let $\alpha\in \Spec(A)$ be an $M$-prime cone of $A$ and let $A/\supp(\alpha)$ is a discrete extension of $M$, then $\alpha\in \Spec(A)$ is an $M$-discrete prime cone of $A$.
\end{lemma}
\begin{proof}
Let $\alpha$ be an $M$-discrete prime cone of $A$ and let $f\in A$. Let $\bar{f}\in A/\supp(\alpha)$ with $\bar{f}>0$ in $A/\supp(\alpha)$. This means that $f\in\alpha$ and $f\notin \supp(\alpha)$, so $f-1 \in\alpha$, so $\bar{f}-1\geq 0$ in $A/\supp(\alpha)$. Also suppose that $a\in M$ and $a>0$, then $a\in M^{\geq0}$ and items (i), (ii) of Definition \ref{dpc} imply that $a\in\alpha$, $a\notin \supp(\alpha)$, hence $a>0$ in $A/\supp(\alpha)$. Thus $A/\supp(\alpha)$ is a discrete extension of $M$.

For the converse let $A/\supp(\alpha)$ is a discrete extension of $M$. We check that $\alpha$ satisfies the requirements of Definition \ref{dpc}. Let $a\in M^{>0}$, then $a>0$ in $A/\supp(\alpha)$ which means that $a\in\alpha$, $a\notin \supp(\alpha)$. So $M^{\geq0}\subseteq\alpha$ and $M^{\geq0}\cap \supp(\alpha)=\{0\}$. Now suppose that $f\in A$, $f\in\alpha$, $f\notin\supp(\alpha)$. It follows that $\bar{f}>0$ in $A/\supp(\alpha)$. Also discreteness of $A/\supp(\alpha)$ implies that $\bar{f}-1\geq0$. So $f-1 \in\alpha$. Thus $\alpha$ is an $M$-discrete prime cone of $A$.
\end{proof}
\begin{definition}
Let $M$ be a discretely ordered ring and let $A$ be a ring with $M\subseteq A$. We say that $\alpha\in \Spec(A)$ is an $M$-\textsl{discrete ordering} of $A$ if $\alpha$ is an $M$-\textsl discrete prime cone of $A$ and $\supp(\alpha)=\{0\}$.
\end{definition}
\begin{remark}
If $\alpha\in \Spec(A)$ is an $M$-discrete ordering of $A$, then $A$ has an ordering which makes $A$ a discrete extension of $M$.
\end{remark}
Now for a discretely ordered ring $M$ and a ring $A\supseteq M$, we can conclude that
\begin{corollary}
There is a bijective correspondence between the set of all $M$-discrete prime cones of $A$ and the set of all pairs $(\mathfrak{p},<)$ where $\mathfrak{p}$ is a real prime ideal of $A$ with $\mathfrak{p}\cap M=\{0\}$ and $(A/\mathfrak{p},<)$ is a discrete extension of $M$. Also as a special case, there is a bijective correspondence between the set of all $M$-discrete orderings of $A$ and the set of all discrete orderings of $A$ which extend the ordering of $M$.
\end{corollary}

Suppose $R$ is a real closed field with $R\supseteq M$ where $M$ is a discretely ordered ring. The inclusion homomorphism $i\colon M[\bar{X}]\rightarrow R[\bar{X}]$ induces a continuous map $\pi\colon\Spec(R[\bar{X}])\rightarrow\Spec(M[\bar{X}])$. For $\alpha\in \Spec(R[\bar{X}])$ we have $\pi(\alpha)=\alpha\cap M[\bar{X}]$. We denote $\pi(\alpha)=\alpha_{M}$. It is easily seen that $\supp(\alpha_{M})=\supp(\alpha)\cap M[\bar{X}]$. The next proposition will enable us to translate our diophantine problems into the real algebraic geometric terms, but to prove it we need the following two lemmas.
\begin{lemma}\label{sur}
Let $R_1\subseteq R_2$ be two real closed fields, then the natural continuous map $\pi\colon \Spec(R_2[\bar{X}])\rightarrow\Spec(R_1[\bar{X}])$ is surjective.
\end{lemma}
\begin{proof}
Let $\alpha\in\Spec(R[\bar{X}])$ and there is no $\beta\in\Spec(R_2[\bar{X}])$ including $\alpha$. It follows from the formal Positivstellensatz (\cite{BCR}, Proposition 4.4.1) that there are $p(\bar{X}),q_1(\bar{X}),\dots,q_r(\bar{X})$ in $\sum R_2[\bar{X}]^2$ and $b_1(\bar{X}),\dots,b_r(\bar{X})\in\alpha$ such that $p(\bar{X})+q_1(\bar{X})b_1(\bar{X})+\dots+q_r(\bar{X})b_r(\bar{X})=-1$. Now by quantifying over the coefficients of $p(\bar{X})$ and $q_i(\bar{X}), i=1,\dots,r$ and using Tarski's transfer theorem we can assume that there are $p^{*}(\bar{X})$, $q_1^{*}(\bar{X})$, $\dots$, $q_r^{*}(\bar{X})$ in $\sum R_1[\bar{X}]^2$ such that $p^{*}(\bar{X})+q_1^{*}(\bar{X})b_1(\bar{X})+\dots+q_r^{*}(\bar{X})b_r(\bar{X})=-1$, which is a contradiction.
\end{proof}
\begin{lemma}\label{sur2}
Let $(A,<)$ be an ordered subring of a real closed field $R$ and let $\pi\colon \Spec(R[\bar{X}])$ $\rightarrow\Spec(A[\bar{X}])$ be the natural map induced by the inclusion map $i\colon A[\bar{X}]\rightarrow R[\bar{X}]$. Suppose $\alpha\in \Spec(A[\bar{X}])$ is an $A$-prime cone. Then $\alpha$ is in the image of $\pi$.
\end{lemma}
\begin{proof}
We assume that $R$ is the real closure of $A$. Then the general case follows from Lemma \ref{sur}. Since $\alpha$ is an $A$-prime cone, we have $A^{\geq 0}\subseteq\alpha$ and $A^{\geq 0}\cap \supp(\alpha)=\{0\}$, from which we conclude that the natural homomorphisms
\[
M[\bar{X}]\rightarrow M[\bar{X}]/ \supp(\alpha)\rightarrow k(\supp(\alpha)) \rightarrow k(\alpha),
\]
induce an order-preserving injective homomorphism $M\rightarrow k(\alpha)$ which itself induces an order-preserving injective homomorphism $R\rightarrow k(\alpha)$. Now observe that image of any $f\in\alpha$ under the composition of homomorphisms:
\[
R[\bar{X}]\rightarrow R[X_1(\alpha),\dots,X_n(\alpha)]\rightarrow k(\alpha)
\]
is $\geq0$. Hence Positivstellensatz ensures that $\alpha$ can be extended to a prime cone of $R[\bar{X}]$.
\end{proof}
\begin{definition}
We say that $\alpha\in\Spec(R[\bar{X}])$ is an \textsl{$M$-arithmetical point} of $\Spec(R[\bar{X}])$ if $\alpha_{M}$ is an $M$-discrete prime cone of $M[\bar{X}]$.
\end{definition}
\begin{proposition}\label{p}
Let $M$ be a discretely ordered ring and $R$ be a real closed field with $R\supset M$. Suppose $\alpha\in\Spec(R[\bar{X}])$ is an $M$-arithmetical point of $\Spec(R[\bar{X}])$, then $M[X_1(\alpha),\dots,$ $X_n(\alpha)]$ is a discretely ordered subring of $k(\alpha)$. Conversely let $\beta\in \Spec(M[\bar{\bar{X}}])$ be an $M$-prime cone of $M[\bar{X}]$ such that $M[X_1(\beta),\dots,$ $X_n(\beta)]$ is a discretely ordered subring of $k(\beta)$, then there is an $M$-arithmetical point $\alpha$ of $\Spec(R[\bar{X}])$ such that $\alpha_{M}=\beta$.
\end{proposition}
\begin{proof}
Let $\alpha\in\Spec(R[\bar{X}])$ be an $M$-arithmetical point of $\Spec(R[\bar{X}])$. There is a canonical homomorphism $j\colon M[\bar{X}]\rightarrow R[\bar{X}]/\supp(\alpha)$. Note that for $f(\bar{X})\in M[\bar{X}]$, $j(f(\bar{X}))=f(X_1(\alpha),\dots,$ $X_n(\alpha))$. So $\ker j=M[\bar{X}]\cap \supp(\alpha)= \supp(\alpha_M)$. It follows that $M[\bar{X}]/ \supp(\alpha_M)\cong \textrm{Im}(j)=M[X_1(\beta),\dots,$ $X_n(\beta)]$ is a discrete extension of $M$.

Conversely, let $\beta$ be an $M$-prime cone of $M[\bar{X}]$ such that $M[X_1(\beta),\dots,$ $X_n(\beta)]$ is a discretely ordered subring of $k(\beta)$. Let $f\in\beta$, $f\notin \supp(\beta)$, then $f>0$ in $M[\bar{X}]/ \supp(\beta)\cong M[X_1(\beta),\dots,$ $X_n(\beta)]$, thus $f-1\geq0$ and so $f-1\in\alpha$. Therefore $\beta$ is an $M$-discrete prime cone of $M[\bar{X}]$. Also by Lemma \ref{sur2}, there is $\alpha\in \Spec(R[\bar{X}])$ such that $\alpha_{M}=\beta$.
\end{proof}
Now we are ready to state our geometric translation of the mentioned diophantine problems.
\begin{theorem}
Let $M$ be a discretely ordered subring of a real closed field $R$. Suppose $f(\bar{X})\in M[\bar{X}]$ and $V\subseteq R^{n}$ is the real algebraic set defined by $f(\bar{X})=0$. Then $f(\bar{X})=0$ has a solution in a discretely ordered ring $N\supseteq M$ iff $\widetilde{V}\subseteq \Spec(R[\bar{X}])$ has an $M$-arithmetical point.
\end{theorem}
\begin{proof}
Let $\alpha\in\widetilde{V}$ is an $M$-arithmetical point, then $f(X_1(\alpha),\dots,X_n(\alpha))=0$ holds true in $k(\alpha)$ and consequently in $M[X_1(\alpha),\dots,$ $X_n(\alpha)]$. Also by Proposition \ref{p} $M[X_1(\alpha),\dots,$ $X_n(\alpha)]$ is a discretely ordered ring.

Let $f(\bar{X})\in M[\bar{X}]$ and $f(\bar{X})=0$ has a solution $\bar{a}=(a_1,\dots,a_n)$ in a discretely ordered ring $N\supseteq M$. Let $\mathfrak{p}=\{g(\bar{X})\in M[\bar{X}]; g(\bar{a})=0\}$, then it is easy to see that $\mathfrak{p}$ is a real prime ideal of $M[\bar{X}]$. Consider the homomorphism $\varphi\colon M[\bar{X}]\rightarrow N$ which sends $g(\bar{X})$ to $g(\bar{a})$. Clearly $\ker\varphi=\mathfrak{p}$ and consequently $\textrm{Im}(\varphi)\cong M[\bar{X}]/\mathfrak{p}$ inherits a discrete ordering $<$ from $N$. Set $\beta=\varphi^{-1}((M[\bar{X}]/\mathfrak{p})^{\geq0})$, thus $\beta$ is a prime cone of $M[\bar{X}]$ with $\supp(\alpha)=\mathfrak{p}$. Therefore $M[X_1(\beta),\dots,$ $X_n(\beta)]\cong M[\bar{X}]/\supp(\beta)\cong M[\bar{X}]/\mathfrak{p}$ is a discretely ordered ring and according to Proposition \ref{p} there is an $M$-arithmetical point $\alpha$ of $\Spec(R[\bar{X}])$ such that $\alpha_{M}=\beta$. Recall that $f(\bar{X})\in\mathfrak{p}= \supp(\beta)=\supp(\alpha)\cap M[\bar{X}]$. It follows that $f(X_1(\alpha),\dots,X_n(\alpha))=0$ holds true in $k(\alpha)$ which means that $\alpha\in\widetilde{V}$.
\end{proof}
Now we have tools to talk about the discrete orderings of $M[\bar{X}]$ in geometric terms.
\begin{definition}
Let $M$ be a discretely ordered subring of a real closed field $R$. We say that $\alpha\in \Spec(R[\bar{X}])$ is a transcendental $M$-arithmetical point of $\Spec(R[\bar{X}])$ if $\alpha$ does not lie on any $\widetilde{V}\subseteq\Spec(R[\bar{X}])$ where $V$ is proper real algebraic subset of $R^{n}$defined over $M$.
\end{definition}
\begin{proposition}
Let $\alpha$ be a transcendental $M$-arithmetical point of $\Spec(R[\bar{X}])$, then $\supp(\alpha_{M})$ $=\{0\}$, i.e., $\alpha_{M}$ is a positive cone of a discrete ordering of $M[\bar{X}]$. Conversely let $\beta$ be a positive cone of a discrete ordering of $M[\bar{X}]$, then there is a transcendental $M$-arithmetical point $\alpha$ of $\Spec(R[\bar{X}])$ such that $\alpha_{M}=\beta$.
\end{proposition}
\begin{proof}
Left to the reader.
\end{proof}
\section{Main Theorem}
In this section we prove the main theorem of this paper, but before that we need a lemma and a proposition.
\begin{lemma}\label{semialgebraic}
Let $R\subseteq R_1\subseteq R_2$ be real closed fields. Let $\mathcal{S}$ be a semialgebraic formula over $R$. Suppose that $\alpha_{2}\in \Spec(R_2[X_1,\dots,X_n])$ and $\alpha_1=\alpha_2\cap R_1[X_1,\dots,X_n]$. Then $\alpha_2\in\widetilde{\mathcal{S}(R_2)}$ iff $\alpha_1\in\widetilde{\mathcal{S}(R_1)}$.
\end{lemma}
\begin{proof}
We have the injective order-preserving $R$-homomorphism
\begin{center}
$R_1[\bar{X}]/ \supp(\alpha_1)\rightarrow R_2[\bar{X}]/ \supp(\alpha_2)$,
\end{center}
sending $X_i(\alpha_1)$ to $X_i(\alpha_2)$, $1\leq i\leq n$, which extends to an injective order preserving $R$-homomorphism
$k_{R_1}(\alpha_1)\rightarrow k_{R_2}(\alpha_2)$.
Therefore $\alpha_1\in\widetilde{\mathcal{S}(R_1)}$ iff $\mathcal{S}(X_1(\alpha_1),\dots,X_n(\alpha_1))$ holds true in $k_{R_1}(\alpha_1)$ iff
$\mathcal{S}(X_1(\alpha_2),\dots,X_n(\alpha_2))$ holds true in $k_{R_2}(\alpha_2)$ iff
$\alpha_2\in\widetilde{\mathcal{S}(R_2)}$.
\end{proof}
\begin{proposition}\label{sphere}
Let $\alpha\in \Spec(R[\bar{X}])$ and $\pi\colon \Spec(k(\alpha)[\bar{X}])\rightarrow\Spec(R[\bar{X}])$ be the canonical projection and $r\in R$. Suppose that
\[
S_{\alpha,r}=\big{\{}\bar{x}\in k(\alpha)^{n}\,|\sum_{i=1}^{n}(x_i-X_i(\alpha))^{2}\leq r^{2}\big{\}},
\]
then $\pi(\widetilde{S_{\alpha,r}})\subseteq\mathbb{B}(\alpha,r)$.
\end{proposition}
\begin{proof}
By way of contradiction let $\beta\in\widetilde{S_{\alpha,r}}$ and $\pi(\beta)=\beta_1$ with $\mu(\alpha,\beta_1)>r$, then there are semi-algebraic subsets $A,B\subseteq R^{n}$ such that $\alpha\in\widetilde{A}$, $\beta_1\in\widetilde B$ and $d(A,B)>r$ (see Section \ref{prelim}). This implies that
\[
R\models\forall \bar{x}\in A\,\, \forall \bar{y}\in B\,\, ||\bar{x}-\bar{y}||>r.
\]
It follows from $\beta_1\in\widetilde B$ that $k_R(\beta_1)\models B\big{(}X_1(\beta_1),\dots,X_{n}(\beta_1)\big{)}$. Since as ordered fields we have the following inclusion
\[
k_R(\beta_1)=RC(F(R[\bar{X}]/\supp(\beta_1)))\subseteq RC(F(k(\alpha)[\bar{X}]/\supp(\beta)))=k_{k(\alpha)}(\beta),
\]
we deduce $k_{k(\alpha)}(\beta)\models B\big{(}X_1(\beta),\dots,X_{n}(\beta)\big{)}$. Then by using Tarski's transfer theorem we get
\[
k_{k(\alpha)}(\beta)\models \forall \bar{x}\in A\,\, \sum_{i=1}^{n}(X_{i}(\beta)-X_{i})^{2}>r^2.
\]
Also it follows from $\alpha\in\widetilde{A}$ that $k(\alpha)\models A(X_1(\alpha),\dots,X_n(\alpha))$. Therefore we have
\[
k_{k(\alpha)}(\beta)\models \sum_{i=1}^{n}(X_{i}(\beta)-X_{i}(\alpha))^{2}>r^2,
\]
which contradicts $\beta\in\widetilde{S_{\alpha,r}}$.
\end{proof}
\begin{theorem}\label{main}
Let $M$ be a discretely ordered subring of a real closed field $R$. Suppose that $r\in R_{fin}^{>0}$ is non-infinitesimal and that $\mathbb{B}(\alpha,t)\cap\widetilde{H}_{\bar{a}}=\emptyset$ for every $\bar{a}\in M^{n}$ and any finite $t\in R$. Then there is $\gamma\in\mathbb{B}(\alpha,r)$ such that $\gamma$ is a transcendental $M$-arithmetical point of $\Spec(R[\bar{X}])$.
\end{theorem}
\begin{proof}
Let $\bar{X}=(X_1,\dots,X_n), \bar{T}=(T_1,\dots,T_n)$ and let $k$ be a positive integer and $H[\bar{X},\bar{T}]\in R[\bar{X},\bar{T}]$. We first define the following operators.\\
\begin{itemize}
\item[(i)] $\nabla_{\!\!\bar{X}} (H[\bar{X},\bar{T}]) =(\frac{\partial H[\bar{X},\bar{T}]}{\partial X_1},\ldots,\frac{\partial H[\bar{X},\bar{T}]}{\partial X_n})$,
\item[(ii)] $(\nabla_{\!\!\bar{X}}\boldsymbol{\cdot} \bar{T})(H[\bar{X},\bar{T}])=(\frac{\partial H[\bar{X},\bar{T}]}{\partial X_1},\ldots,\frac{\partial H[\bar{X},\bar{T}]}{\partial X_n})\boldsymbol{\cdot} \bar{T}=\displaystyle\sum_{i=1}^{n}\frac{\partial H[\bar{X},\bar{T}]}{\partial X_i}T_i$,
\item[(iii)] $(\nabla^{1}_{\!\!\bar{X}}\boldsymbol{\cdot}\bar{T})(H[\bar{X},\bar{T}])=(\nabla_{\!\!\bar{X}}\boldsymbol{\cdot}\bar{T})(H[\bar{X},\bar{T}])$,\\
\item[(iv)] $(\nabla^{k+1}_{\!\!\bar{X}}\boldsymbol{\cdot}\bar{T})(H[\bar{X},\bar{T}])=(\nabla^{1}_{\!\!\bar{X}}\boldsymbol{\cdot}\bar{T})
    \big{[}(\nabla^{k}_{\!\!\bar{X}}\boldsymbol{\cdot}\bar{T})(H[\bar{X},\bar{T}])
    \big{]}$.\\
\end{itemize}
Suppose that $F(X_1,\dots,X_n)\in M[X_1,\dots,X_n]$ is a non-constant polynomial. We separate two cases: (i) $\deg{F}\geq 2$, (ii) $\deg{F}=1$. We first deal with the first case. The second case will be similar to the last stage of the first case when we have done the reduction. So suppose that $\deg{F}=m+1, m\geq 1$. For $1\leq k\leq m$, put
\[
G_k(\bar{X},\bar{T})=(\nabla^{k}_{\!\!\bar{X}}\boldsymbol{\cdot}\bar{T})F(\bar{X}).
\]
We claim that for $1\leq k\leq m$, the polynomials $G_k(\bar{X},\bar{T})$ are non-zero polynomials with $\deg_{\bar{X}} G_k=m-k+1$. Let $D\subseteq\omega_{*}^{n}$ be such that $F(\bar{X})=\sum_{\bar{v}\in D}a_{\bar{v}} X_1^{v_1}\cdots X_n^{v_n}$ where $a_{\bar{v}}\neq 0$ for $\bar{v}\in D$. The claim will be proved as an immediate consequence of the two easily verified facts below. For any $\bar{v}\in D$ set
\[
D_{\bar{v},k}=\{\bar{w}\in\omega_{*}^{n}\,|\,v_i\geq w_i, \sum_{i=1}^{n}(v_i-w_i)=k\}.
\]

\textbf{Fact 1.} Fix $\bar{v}\in D$, then for every $\bar{w}\in D_{\bar{v},k}$ there is $b_{\bar{v},\bar{w}}\neq 0$ in $M$ such that for every non-zero $a\in M$, the following equation holds
\[
(\nabla^{k}_{\!\!\bar{X}}\boldsymbol{\cdot}\bar{T})(a\,X_1^{v_1}\cdots X_n^{v_n})=\sum_{\substack{\bar{w}\in D_{v,k}\\ \bar{w}+\bar{l}=\bar{v}}} b_{\bar{v},\bar{w}}T^{l_1}_{1}\cdots T^{l_n}_{n}\,X_1^{w_1}\cdots X_n^{w_n}.
\]
This is proved by induction on $k$ and also observing that there is no change of signs in the coefficients.\\

\textbf{Fact 2.} Let $D_k=\displaystyle\bigcup _{\bar{v}\in D} D_{\bar{v},k}$, then for every $\bar{w}\in D_k$ there is a nonzero polynomial $P_{k,\bar{w}}(T)\in M[\bar{T}]$ such that
\begin{equation}\label{coefficients}
G_k(\bar{X},\bar{T})=(\nabla_{\!\!\bar{X}}^{k}\boldsymbol{\cdot}\bar{T})(F(\bar{X}))=\sum_{\bar{w}\in D_k} P_{k,\bar{w}}(\bar{T})X_1^{w_1}\cdots X_n^{w_n}.
\end{equation}
where
\[
P_{k,\bar{w}}(\bar{T})=\sum_{\substack{\bar{v}\in D\\\bar{w}\in D_{\bar{v},k}\\ \bar{w}+\bar{l}=\bar{v}}} b_{\bar{v},\bar{w}} T^{l_1}_{1}\cdots T^{l_n}_{n}.
\]
Note that if $\bar{w}\in D_{\bar{v},k}$, $\bar{w}^{'}\in D_{\bar{v}^{'},k}$ and $\bar{w}=\bar{w}^{'}$, then $\bar{l}\neq \bar{l}^{'}$ where $\bar{l}+\bar{w}=\bar{v}$ and $\bar{l}^{'}+\bar{w}^{'}=\bar{v}^{'}$. This means that the monomials $b_{\bar{v},\bar{w}} T^{l_1}_{1}\cdots T^{l_n}_{n}$ do not remove each other when $\bar{v}$ varies through $D$ and consequently we obtain a nonzero $P_{k,\bar{w}}(\bar{T})$. This proves the claim. Now related to $F(\bar{X})$, we define
\[
U^1_F=\big{\{}\bar{t}\in k(\alpha)^{n}\,|\, P_{k,\bar{w}}(t_1,\dots,t_n)\neq 0,\, 1\leq k\leq m,\,\bar{w}\in D_k\big{\}}.
\]
Also from the above consideration, it follows that we can write $G_k(\bar{X},\bar{T})$ as
\[
G_k(\bar{X},\bar{T})=P_0(\bar{T})+P_1(\bar{T})X_1+\cdots+P_n(\bar{T})X_n,
\]
where $P_i(\bar{T})\in M[\bar{T}], 0\leq i\leq n$ and not all of $P_1(\bar{T})\ldots,P_n(\bar{T})$ are zero. Also we set
\[
U^2_F=\big{\{}\bar{t}\in k(\alpha)^{n}\,|\, \sum_{i=1}^{n}t_iP_i(t_1,\dots,t_n)\neq 0\big{\}}.
\]
Clearly $U^1_F$, $U^2_F$ are two nonempty Zariski open subsets of the affine space $\mathbb{A}^n_{k(\alpha)}$. So does
\[
U_F:=U^1_F\cap U^2_F.
\]
By the density of the $\mathbb{Q}$-rational points in $\mathbb{A}^n_{k(\alpha)}$ with the Zariski topology we can choose
\[
\bar{q}=(q_1,\dots,q_n)\in\mathbb{Q}^n\cap U_F.
\]
Let $S_{\alpha,r}$ be
\[
S_{\alpha,r}=\big{\{}\bar{x}\in k(\alpha)^{n}\,|\sum_{i=1}^{n}(x_i-X_i(\alpha))^{2}\leq r^{2}\big{\}}.
\]
For $P=(p_1,\dots,p_n)\in k(\alpha)^n$, we say that $P$ is $F$-\emph{infinite} if for every $k\in\mathbb{N}$ we have
\[
k(\alpha)\models |F(p_1,\dots,p_n)|>k.
\]
Put $P=(X_1(\alpha),\dots,X_n(\alpha))\in k(\alpha)^n$ and let $Q\in S_{\alpha,r}$ be such that the vector $\overrightarrow{PQ}$ has non-infinitesimal length and is parallel to $\overrightarrow{q}$, i.e., the vector with tail=origin and head=$\bar{q}$. We claim (our second claim) that there is $N\in\mathbb{N}$ such for every $N$ points $P_1,\dots,P_N$ lying on the segment $\overline{PQ}$ with the following properties:
\begin{itemize}
\item[(i)] $P_1=P,\,P_N=Q$,
\item[(ii)] $||\overrightarrow{P_iP}_{i+1}||$ is non-infinitesimal,
\item[(iii)] $P_i$ lies between $P_{i-1}$ and $P_{i+1}$,
\end{itemize}
there is $P\in\{P_1,\dots,P_n\}$ so that $P$ if $F$-infinite. To prove the claim we first define a finite sequence of integers $N_1,\dots,N_m$ as follows. Put $N_1=[\frac{N-1}{2}]$, $N_{i+1}=[\frac{N_i-1}{2}]$ for $1\leq i\leq m-1$ and take $N$ large enough so that $N_m\geq 2$. Now by way of contradiction suppose that all $F(P_i)$'s are finite. Write $\overrightarrow{P_iP}_{i+1}=\lambda_i\overrightarrow{q}$, $\lambda_i\in k(\alpha)$. It is easy to see that $\lambda_i$ is finite and non-infinitesimal. Using the higher dimensional mean value theorem for real closed fields we get $Q_i\in\overline{P_iP}_{i+1}$ such that
\begin{eqnarray*}
F(P_{i+1})-F(P_i)=
\nabla F(Q_i)\boldsymbol{\cdot}\overrightarrow{P_iP}_{i+1}
&=&\lambda_i\,\big{(}\nabla F(Q_i)\boldsymbol{\cdot}\overrightarrow{q}\big{)}\\
&=&\lambda_i\,\big{(}\sum_{j=1}^{n}\frac{\partial F}{\partial X_j}(Q_i)\,.\, q_j\big{)}\\
&=&\lambda_i\,\, G_1(Q_i,\bar{q}).
 \end{eqnarray*}
Since $F(P_i),F(P_{i+1}),\lambda_i$ are finite and $\lambda_i$ is not infinitesimal, it follows that $G_1(Q_i,q)$, $i=1,\dots,N-1$ is finite. Observe that $Q_1,Q_3,Q_5,\dots$ are not infinitesimally close to each other and we have $N_1=[\frac{N-1}{2}]$ many of them. So we can repeat the argument for the polynomial $G_1(\bar{X},\bar{q})\in R[\bar{X}]$ and points $Q_1,Q_3,Q_5,\dots$. Renaming the points as $Q_1,Q_2,Q_3,\dots$ and writing $\overrightarrow{Q_iQ}_{i+1}=\theta_i\overrightarrow{q}$, $\theta_i\in k(\alpha)$, we obtain points $R_i\in\overline{Q_iQ}_{i+1}$ such that
\begin{eqnarray*}
G_1(Q_{i+1},\bar{q})-G_1(Q_i,\bar{q})=
\nabla G_1(R_{i},\bar{q})\boldsymbol{\cdot}\overrightarrow{Q_iQ}_{i+1}
&=&\theta_i\,\big{(}\nabla G_1(R_{i},\bar{q})\boldsymbol{\cdot}\overrightarrow{q}\big{)}\\
&=&\theta_i\,\big{(}\sum_{j=1}^{n}\frac{\partial G_1}{\partial X_j}(R_i)\,.\, q_j\big{)}\\
&=&\theta_i\,\, G_2(R_i,\bar{q}).
\end{eqnarray*}
where $\theta_i$ are finite and infinitesimal. Again it follows that $G_2(R_i,\bar{q})$, $i=1,\dots,N_1-1$ is finite as well as $R_1,R_3,R_5,\dots$ are not infinitesimally close to each other and we have $N_2=[\frac{N_1-1}{2}]$ many of them. Thus after $m$ times repeating the argument and considering the fact that $N_m\geq2$ we finally obtain two points $\bar{a}=(a_1,\dots,a_n)$ and  $\bar{b}=(b_1,\dots,b_n)$ on the segment $\overline{PQ}$ such that they are not infinitesimally close to each other and $G_m(\bar{a},\bar{q})$ and $G_m(\bar{b},\bar{q})$ are finite. Recall the representation $G_m(\bar{X},\bar{T})=P_0(\bar{T})+P_1(\bar{T})X_1+\dots+P_n(\bar{T})X_n$ with $P_i(\bar{T})\in M[\bar{T}]$, then by the choice of $\bar{q}$, not all of $P_1(\bar{q}),\dots,P_n(\bar{q})$ are zero. Note that what makes this reduction process terminates to the linear polynomial $G_m(\bar{X},\bar{T})$ is the fact that the previous $G_i(\bar{X},\bar{T})$'s and $G_i(\bar{X},\bar{q})$'s all were non-zero polynomials. Now at this point we use some Euclidean geometry. Working in $k(\alpha)^{n+1}$ we let
\[
A_1=(\bar{a},0)\,\,\,\,\, B_1=(\bar{b},0)
\]
\[
A_2=(\bar{a},G_m(\bar{a},\bar{q}))\,\,\,\,\, B_2=(\bar{b},G_m(\bar{b},\bar{q})).
\]
The lines $A_1A_2$ and $B_1B_2$ are parallel to each other so let $\mathcal{P}$ be the plane passing through the points $A_1,A_2,B_1,B_2$. We verify the two following facts:\\

\textbf{Fact 3.} $||\overline{A_1A_2}||\neq ||\overline{B_1B_2}||$.\\

If $||\overline{A_1A_2}||=||\overline{B_1B_2}||$, then we would have
\[
P_1(\bar{q})(a_1-b_1)+\cdots+P_1(\bar{q})(a_n-b_n)=0.
\]
Letting $(a_i-b_i)=\delta q_i, \delta\in k(\alpha)$ we get
\[
P_1(\bar{q})q_1+\cdots+P_1(\bar{q})q_n=0,
\]
which violates $\bar{q}\in U_F^2$.

Now without loss of generality we can assume $||\overline{A_1A_2}||>||\overline{B_1B_2}||$.\\

\textbf{Fact 4.} $||\overline{A_1A_2}||-||\overline{B_1B_2}||$ is not infinitesimal.\\

Notice that $\delta$ in the proof of Fact 3. is finite non-infinitesimal, so if $||\overline{A_1A_2}||-||\overline{B_1B_2}||$ is infinitesimal, then we would have
\[
P_1(\bar{q})q_1+\cdots+P_n(\bar{q})q_n=\textrm{infinitesimal}.
\]
Let $m$ be a suitable integer that $mP_1(\bar{q})q_1,\ldots,mP_n(\bar{q})q_n$ be all in $M$, then
\[
mP_1(\bar{q})q_1+\cdots+mP_n(\bar{q})q_n=m\times(\textrm{infinitesimal})=\textrm{infinitesimal},
\]
which violates the discreteness of $M$. This prove Fact 4.

Now let the lines $A_1B_1$ and $A_2B_2$ intersect each other at the point $C\in\mathcal{P}$. By similarity of the two triangles $\triangle{A_1A_2C}$ and $\triangle{B_1B_2C}$ we have
\[
\frac{||\overline{B_1C}||}{||\overline{A_1C}||}=\frac{||\overline{B_1B_2}||}{||\overline{A_1A_2}||}
\]
which implies that
\[
\frac{||\overline{B_1C}||}{||\overline{A_1C}||-||\overline{B_1C}||}=\frac{||\overline{B_1C}||}{||\overline{A_1B_1}||}=
\frac{||\overline{B_1B_2}||}{||\overline{A_1A_2}||-||\overline{B_1B_2}||}.
\]
Recall that $||\overline{A_1B_1}||$ is finite non-infinitesimal, $||\overline{B_1B_2}||$ is finite and Fact 4., so we conclude that  $||\overline{B_1C}||$ is finite. On the other hand being $C$ on the line $A_1B_1$ implies that the $(n+1)$-th coordinate of $C$ in $k(\alpha)^{n+1}$ is zero.
Also being $C$ on the line $A_2B_2$ implies that the point $C$ lies on the hyperplane
\[
P_1(\bar{q})X_1+\cdots+P_n(\bar{q})X_n=X_{n+1}.
\]
Thus writing $C=(c_1,\dots,c_n,0)$ we infer that
\[
P_1(\bar{q})c_1+\cdots+P_n(\bar{q})c_n=0.
\]
So
\[
mP_1(\bar{q})c_1+\cdots+mP_n(\bar{q})c_n=0.
\]
Now regarding $\bar{c}=\gamma=(c_1,\dots,c_n)$ as an element of $\Spec(k(\alpha)[X_1,\dots,X_n])$ we get that $\gamma\in \widetilde{H_{\bar{d}}}$ where $\bar{d}=(mP_1(\bar{q}),\dots,mP_n(\bar{q}))\in M^n$. Let
\[
\pi\colon\Spec(k(\alpha)[X_1,\dots,X_n])\longrightarrow\Spec(R[X_1,\dots,X_n])
\]
be the canonical projection, so we have $\pi(\gamma)\in \widetilde{H_{\bar{d}}}(R)$. Also finiteness of $||\overline{B_1C}||$ implies that there is $t\in R_{fin}$ such that
\[
\sum_{i=1}^{n}(c_i-X_i(\alpha))^{2}\leq t^{2},
\]
so $\gamma\in\widetilde{S_{\alpha,t}}$ which implies that $\pi(\gamma)\in\mathbb{B}(\alpha,t)$. Hence we get $\pi(\gamma)\in\mathbb{B}(\alpha,t)\cap\widetilde{H_{\bar{d}}}(R)$ which violates the hypothesis of the theorem. This proves the second claim for the case $\deg{F}\geq 2$. For the case $\deg F=1$, let $F(\bar{X})=p_0+p_1 X_1+\cdots+p_n X_n$ where $p_i\in M$ for $0\leq i\leq n$. We treat $F(\bar{X})$ exactly as $G_m(\bar{X},\bar{q})$ by putting $p_i=P_i(\bar{q})$, $N=2$,
\[
U_F=\big{\{}\bar{x}\in k(\alpha)^{n}\,|\, \sum_{i=1}^{n}p_i x_i\neq 0\big{\}},
\]
and repeating the last stage of the above argument. This will prove the second claim.

Observe that by considering $2N+1$ points satisfying the conditions (i), (ii), (iii) above, we can actually find two $F$-infinite points which are not infinitesimally close to each other and in fact by iterating, we can find as many such $F$-infinite points as we want. So if $G(X_1,\dots,X_n)$ is another polynomial in $M[X_1,\dots,X_n]$, then by repeating the argument of the second claim, we can find a $G$-infinite point among those previously found $F$-infinite points (which are sufficiently many) provided that $\bar{q}$ is chosen from $U_{F}\cap U_{G}$. Therefore we can easily generalize the assertion in the second claim for any finite set of polynomials $F_1(\bar{X}),\dots, F_l(\bar{X})\in M[\bar{X}]$ in which we have that there is $N$ so that for any $N$ points $P_1,\dots, P_N$ satisfying the conditions (i), (ii), (iii) above, there is $P\in\{P_1,\dots, P_N\}$ which is $F_i$-infinite for $i=1,\dots,l$. Obviously this time we choose a $\mathbb{Q}$-rational point $\bar{q}$ from $U_{F_1}\cap\dots\cap U_{F_l}$.

Now for every finite set of polynomials $F_1(\bar{X}),\dots, F_l(\bar{X})\in M[\bar{X}]$ and every $k\in\mathbb{N}$, we set
\[
V^k_{F_1,\dots,F_l}=\{\bar{x}\in k(\alpha)^n: \bigwedge_{i=1}^{l} |F_i(\bar{x})|>k\}.
\]
In fact we have obtained
\[
\emptyset\neq\widetilde{S_{\alpha,r}}\cap\widetilde{V}^k_{F_1,\dots,F_l},
\]
for each $k\in\mathbb{N}$. So by compactness of $\Spec(k(\alpha)[\bar{X}])$ in the constructible topology, there is $\gamma^{*}\in\widetilde{S_{\alpha,r}}$ such that $\gamma^{*}\in\widetilde{V}^k_{F_1,\dots,F_l}$ for all $F_1[\bar{X}],\dots,F_l[\bar{X}]\in M[\bar{X}], k\in\mathbb{N}$. We conclude that $\gamma:=\pi(\gamma^{*})\in\mathbb{B}(\alpha,r)$ and also $\gamma\in\widetilde{V}^k_{F_1,\dots,F_l}(R)$ which immediately implies that $M[X_1(\gamma),\dots,X_n(\gamma)]$ is a discretely ordered ring which is  equivalent to saying that $\gamma$ is a transcendental $M$-arithmetical point. This finishes the proof.
\end{proof}

\bibliography{reference}
\bibliographystyle{plain}
\end{document}